[1994/12/01]
\documentclass[12pt, reqno]{amsart}
\usepackage{a4wide}
\usepackage[english, activeacute]{babel}
\usepackage{amsmath,amsthm,amsxtra}
\usepackage{epsfig}
\usepackage{amssymb}
\usepackage{latexsym}
\usepackage{amsfonts}
\usepackage[colorlinks=true]{hyperref}
\hypersetup{linkcolor=red,citecolor=blue,filecolor=dullmagenta,urlcolor=red} 

\title{Breathers and the dynamics of solutions in KdV type equations}
\author{Claudio Mu\~noz}

\author{Gustavo Ponce}
\thanks{CM was partially supported by Fondecyt no. 1150202, Millennium Nucleus Center for Analysis of PDE NC130017, Fondo Basal CMM, and MathAmSud EEQUADD collaboration Math16-01. Part of this work was done while the first author was visiting Fields Institute (Toronto, Canada), as part of the ``Focus Program on Nonlinear Dispersive Partial Differential Equations and Inverse Scattering''.}
\address{CNRS and Departamento de Ingenier\'ia Matem\'atica DIM-CMM UMI 2807-CNRS \\ Universidad de Chile, Santiago, Chile}
\email{cmunoz@dim.uchile.cl}
\address{Department of Mathematics \\ University of California-Santa Barbara, CA 93106-USA}
\email{ponce@math.ucsb.edu}
\date{\today}

\chardef\bslash=`\\ 

\newtheorem{thm}{Theorem}[section]
\newtheorem{cor}[thm]{Corollary}
\newtheorem{lem}[thm]{Lemma}
\newtheorem{prop}[thm]{Proposition}

\theoremstyle{remark}
\newtheorem{rem}{Remark}[section]

\numberwithin{equation}{section}

\newcommand{\R}{\mathbb{R}}

\newcommand{\Z}{\mathbb{Z}}

\newcommand{\la}{\lambda}

\newcommand{\al}{\alpha}
\newcommand{\bt}{\beta}
\newcommand{\ga}{\gamma}

\newcommand{\sech}{\operatorname{sech}}

\newcommand{\be}{\begin{equation}}
\newcommand{\ee}{\end{equation}}
\newcommand{\bp}{\begin{proof}}
\newcommand{\ep}{\end{proof}}
\newcommand{\bel}{\begin{equation}\label}
\newcommand{\eeq}{\end{equation}}
\newcommand{\bea}{\begin{eqnarray}}
\newcommand{\eea}{\end{eqnarray}}
\newcommand{\bee}{\begin{eqnarray*}}
\newcommand{\eee}{\end{eqnarray*}}
\newcommand{\ben}{\begin{enumerate}}
\newcommand{\een}{\end{enumerate}}

 \providecommand{\abs}[1]{\lvert#1 \rvert}

\newcommand{\ve}{\varepsilon}

\newcommand{\eval}[2][\right]{\relax
  \ifx#1\right\relax \left.\fi#2#1\rvert}

\let\abs=\envert

\begin{document}
\begin{abstract}
In this paper our first aim is to identify a large class of non-linear functions $\,f(\cdot)\,$ for which the IVP for the generalized Korteweg-de Vries equation does not have breathers or \lq\lq small" breathers solutions. Also we  prove that all small, uniformly in time $L^1\cap H^1$ bounded solutions to KdV and related perturbations must converge to zero, as time goes to infinity,  locally in an increasing-in-time region of space of order $t^{1/2}$ around any compact set in space. This set is included in the linearly dominated dispersive region $x\ll t$.  Moreover, we prove this result independently of the well-known supercritical character of KdV scattering. 
In particular, no standing breather-like nor solitary wave structures exists in this particular regime. 
\end{abstract}
\maketitle \markboth{KdV decay estimates}{C. Mu\~noz-G. Ponce}
\renewcommand{\sectionmark}[1]{}

\section{Introduction and main results}


This work is concerned with the generalized Korteweg-de Vries equation (gKdV) posed in the real line $\R$:
\be\label{gKdV} 
\begin{aligned}
 \partial_t u + \partial_{x}(\partial_x^2 u +f(u)) = 0,\qquad (t,x) \in & ~ \R\times \R.
\end{aligned}
\ee
Here and along this paper $u=u(t,x) \in \R$ is a real-valued function. We will assume that $f=f(s)$ is a focusing\footnote{The case of defocusing nonlinearities is in some sense simpler and can be obtained from the ideas in this paper with somehow easier proofs.}, polynomial-type nonlinearity, in the sense that for $k=2,3,\ldots$ and $f_k:\R \to \R$ of class $C^k$,
\be\label{f(s)}
f(s) = u^k + f_k(s), \quad \quad \lim_{s\to 0} \frac{f_k(s)}{|s|^k} =0.
\ee
Three interesting cases for $f$ as above are 
\ben
\item $k=2$ and $f_2(s)=0$ (the original Korteweg-de Vries (KdV) equation), 
\item $k=2$ and $f_2(s) = \mu s^3$, $\mu \neq 0$ real-valued (the Gardner equation), and finally, 
\item $k=3$ and $f_3(s)=0$ (the modified KdV equation, or mKdV). 
\een
These are well-known integrable models \cite{AS}, and possibly the only gKdV integrable equations, under reasonable hypotheses on the nonlinearity. A rigorous  ``classification result'' involving these nonlinearities and the remaining non-integrable class can be found e.g. in \cite{Munoz}, generalization of previous and foundational work by Martel and Merle \cite{MMcol} for $k=4$ and $f_4(s)=0$. 

\medskip
Formally,  solutions of \eqref{gKdV} satisfy at least three conservation laws:
\begin{equation}
\label {CL}
\begin{aligned}
I_1(u)&:=\int_{-\infty}^{\infty}u(t,x)dx=I_1(u_0),\quad \;\;\;I_2(u):=\int_{-\infty}^{\infty}u^2(t,x)dx=I_2(u_0),\\
\\
I_3(u)&:=\int_{-\infty}^{\infty}\left(\frac12 (\partial_xu)^2 -G(u)\right)(t,x)dx=I_3(u_0),
\end{aligned}
\end{equation}
where
\begin{equation}\label{cl}
G(x)=\int_0^xf(s)ds.
\end{equation}

\medskip

For the simpler pure power case $f(s) =s^k$, the Cauchy problem for \eqref{gKdV} is globally well-posed in $H^1(\R)$ for $k=2,3$ and $4$, with uniform bounds in time for $\|u(t)\|_{H^1}$ thanks to the conservation laws $\,I_2(u)$ and $\,I_3(u)$ in \eqref{CL}, see Kenig-Ponce-Vega \cite{KPV1}. On the other hand, \eqref{gKdV} with $f_k(s)\neq 0$ is globally well-posed if e.g. the initial data is sufficiently small. The Gardner equation is $H^1$ globally well-posed independently of the size of the initial data \cite{KPV1}.

\medskip

It is well-known that  \eqref{gKdV} may have soliton (or solitary wave) solutions of the form
\be\label{Soliton}
u(t,x)= Q_c(x-ct),  \quad c>0,
\ee
where $Q_c$ is a stationary solution to the ODE 
\begin{equation}
\label{solitons}
Q_c''-cQ_c +f(Q_c)=0,\;\;\;\;\;\;\;\;\;\;Q_c\in H^1(\R),
\end{equation}
provided $f$ satisfies standard assumptions. Additionally, it is well-known that both mKdV and Gardner models do have stable breather solutions \cite{AS,AM,AM1,AM2,Alejo}, that is to say, \emph{localized in space solutions which are also periodic in time}, up to the symmetries of the equation. An example of these type of solutions is the mKdV breather: for any $\al,\bt>0$, 
\be\label{breather}
\begin{aligned}
B(t,x) :=  ~2\sqrt{2} \partial_x \arctan \left(\frac{\beta  \sin(\al (x+\delta t))}{\al \cosh(\beta (x+\ga t))} \right),\quad \delta=  \al^2 -3\bt ^2, ~\gamma= 3\al^2 -\bt ^2,
\end{aligned}
\ee
is a solution of mKdV with nontrivial time-periodic behavior, up to the translation symmetries of the equation. Therefore, generalized KdV equations may have both solitary waves and breathers as well, and both classes of solutions do not decay.

\medskip

The Gardner equation  ($f_2(u):= \mu u^3$ here)
\be\label{Gardner}
\partial_t u + \partial_{x}(\partial_x^2 u +  u^2 + \mu u^3) = 0, \quad  \mu>0,
\ee
possesses (stable) breather solutions, see \cite{Alejo} for more details. Indeed,  if  $\al, \bt>0$ are such that $\Delta:=\al^2+\bt^2-\frac2{9\mu}>0$, then
\begin{equation}
\label {BGe}
B(t,x) :=   2\sqrt{\frac 2\mu}\partial_x\arctan\left(\frac{\mathcal G(t,x)}
{\mathcal F(t,x)}\right),
\end{equation}
where
\begin{equation}
\label{BGe2}
\begin{aligned}
\mathcal G & :=   \frac{\bt\sqrt{\al^2+\bt^2}}{\al\sqrt{\Delta}}\sin(\al y_1) -\frac{\sqrt2 \bt[\cosh(\bt y_2)+\sinh(\bt y_2)]}{3\sqrt{\mu}\Delta},\\
\mathcal F & :=   \cosh(\bt y_2)-\frac{\sqrt2 \bt[\al\cos(\al y_1)-\bt\sin(\al y_1)]}{3\sqrt{\mu}\al\sqrt{\al^2+\bt^2}\sqrt{\Delta}},
\end{aligned}
\end{equation}
and $y_1 = x+ \delta t$, $y_2 = x+ \ga t$ , $\delta := \al^2-3\bt^2$, $\ga :=3\al^2-\bt^2$, is a smooth non decaying solution to \eqref{Gardner}.
\medskip

Breather solutions of the IVP \eqref{gKdV} are known only in the previously described  cases, i.e. for the mKdV and for the Gardner equation. Our first goal here is to show that for a large class of nonlinearities $\,f\,$ in \eqref{gKdV} the gKdV does not possess breather solutions (localized in space solutions which are periodic in time). 
\medskip

We recall the standard notation
\begin{equation}
\label{space}
H^{s,r}(\R):= H^s(\R)\cap L^2(|x|^{2r}dx),\;\;\;\;\;\;\;\;\;s,\,r\geq 0.
\end{equation}

\begin{thm}[Nonexistence of breathers]\label{theo:1}
 
Let
\begin{equation}
\label{clas1}
u\in C(\R:H^{4,2}(\R)) 
\end{equation}
be a  non-trivial strong solution of the IVP \eqref{gKdV}. If  $u=u(t,x)$ is a time periodic function with period $\omega$, then
\begin{equation}
\label{A1}
\int_0^{\omega}\,\int_{-\infty}^{\infty} f(u(t,x))dx\,dt=0.
\end{equation}
In particular, if 
\begin{equation}
\label{a1}
\begin{cases}
\begin{aligned}
&f(u)=u^2,\;\;\;\;\;\;\;\text{or}\;\;\;\;\;\;\;f(u)=u^4,\\
&\text{or}\\
&f(u)=u^2\pm\mu u^3+\epsilon^2u^4,\;\;\;\;\;\;\;0\leq\mu<2\epsilon,
\end{aligned}
\end{cases}
\end{equation}
then $u$ cannot be  a  time periodic solution.

\end{thm}

\begin{rem}
 The hypothesis $\,u(t)\in H^{4,2}(\R)$ is just needed to guarantee that $\,xu(t,x)\in L^1(\R)$, so it can be reduced to $\,u(t)\in H^{3^+,(3/2)^+}(\R)$. We recall that in \cite{ILP} the solution flow of the \eqref{gKdV} preserves the class $H^{s,r}(\R)$ only if $s\geq 2r$. 
\end{rem}

\begin{rem}
 The third case in \eqref{a1}, $\,f(u)=u^2\pm\mu u^3+\epsilon^2u^4,\;0\leq\mu<2\epsilon\,$ gives an example of a higher order perturbation of the Gardner equation without breather solutions of any size.
\end{rem}

\begin{rem} 
We recall that global solutions for the IVP \eqref{gKdV} as those in \eqref{clas1} are known to exist provided $\,f(\cdot)\,$
satisfies a growth condition, for example :
\begin{equation}
\label{gra}
\lim_{|s|\to\infty} \frac{f'(s)}{|s|^{4}} = 0,
\end{equation}
and $\,u_0\in H^{4,2}(\R)$, see \cite{Ka}.
\end{rem}

\vskip.2in 
\begin{thm}[Even nonlinearities, general case]\label{theo:2}
 
Let
\begin{equation}
\label{clas2}
u\in C(\R:H^{4,2}(\R)) \cap L^{\infty}(\R: H^1(\R))
\end{equation}
be a  non-trivial strong solution of the IVP \eqref{gKdV}, 
with
\begin{equation}
\label{a2}
f(u)= u^{2k}+f_{2k}(u),\; \;\;\;\;\;\;k\in\Z^+.
\end{equation}
There exists a small $\epsilon=\epsilon(k;f_{2k})>0$ such that if
\begin{equation}
\label{a3}
\sup_{t\in\R}\|u(t)\|_{H^1}<\epsilon,
\end{equation}
then $u=u(t,x)$ is not  a  time periodic solution.
\end{thm}

\vskip.2in 

\begin{rem}
In the case $k=1$, taking $\,f(u)=u^2+\mu u^3 +f_3(u)\,$ we get that the Gardner equation and any \lq\lq small" perturbation of it can not have $H^1$-small breather solutions. For the Gardner equation this is consistence with the properties of the breathers described in \eqref{BGe}-\eqref{BGe2}, see the third remark (Remark \ref{Gardner_remark}), after the statement of Theorem \ref{TH2}.
\end{rem}

\vskip.1in

A second result is the following.

\begin{thm}[$L^2$ critical and supercritical cases]\label{theo:3}
 
Let
\begin{equation}
\label{clas3}
u\in C(\R:H^{1,1/2}(\R)) \cap L^{\infty}(\R: H^1(\R))
\end{equation}
be a non-trivial strong solution of the IVP \eqref{gKdV} with 
\begin{equation}
\label{a5}
f(u)=u^{j}+f_{j}(u), \;\;\;\; \;\;\;\;j=5,6,..
\end{equation}
There exists a small $\epsilon=\epsilon(j;f_{j})>0$ such that if
\begin{equation}
\label{a6}
\sup_{t\in\R}\|u(t)\|_{H^1}<\epsilon,
\end{equation}
then  $u=u(t,x)$ is not  a  time periodic solution.
\end{thm}

\vskip.1in
\begin{rem}  
In  the case $\,j=5\,$ and $\,f_5(u)\equiv 0$, i.e. $\,f(u)=u^5$, the assumption \eqref{a6} can be replaced by a weaker one : the initial data $\,u_0\,$ satisfies
\begin{equation}
\label{seze-data}
\|u_0\|_2\,\leq\,\sqrt[4]{\frac{3}{5}}\,\| Q\|_2,
\end{equation}
where $\,Q\,$ is the solution of \eqref{solitons} with $\,f(u)=u^5$. In this setting the proof follows by combining the argument given below in the proof of Theorem \ref{theo:3} with the best constant $c_6$ for  the Gagliardo-Nirenberg inequality \eqref{GN} with $p=6$ found in \cite{We}, i.e. $c_6=\sqrt[6]{3}\,\|Q\|_2^{-2/3}$. 
\end{rem}

\begin{rem} Restricting ourselves to pure power non-linearity $f(u)=u^k,\,k=2,3,..$ and to the frame where the existence of global solutions are known Theorems \ref{theo:1}-\ref{theo:3} tell us that a breather can only occur in the case $k=3$, i.e. for the MKdV. 
\end{rem}

\vskip.1in

We also consider  nontrivial perturbations of the mKdV equation.

\begin{thm}[Cubic case]\label{theo:4}
 
Let
\begin{equation}
\label{clas4}
u\in C(\R:H^{1,1/2}(\R)) \cap L^{\infty}(\R: H^1(\R))
\end{equation}
be a  non-trivial strong solution of the IVP \eqref{gKdV} with 
\begin{equation}
\label{a7}
f(u)=u^{3}+\beta u^5+f_{5}(u), \;\;\;\;\; \;\;\beta<0,
\end{equation}
and  initial data $u_0$ such that (see \eqref{CL})
\begin{equation}
\label{a8}
I_3(u)=I_3(u_0)\geq 0.
\end{equation}
There exists a small $\epsilon=\epsilon(f_{5};\beta;I_3(u_0))>0$ such that if
\begin{equation}
\label{a9}
\sup_{t\in\R}\|u(t)\|_{H^1}<\epsilon,
\end{equation}
then  $u=u(t,x)$ is not  a  time periodic solution.

The same conclusion applies if $f(\cdot)$ satisfies \eqref{a7} with $\beta>0$ and the initial data $u_0$ holds that $I_3(u_0)\leq 0$.
\end{thm}
\vskip.2in 
\begin{rem}
 The assumption  $\,u(t)\in H^{1,1/2}(\R)$ tells us that $\,xu^2(t)\in L^1(\R)$, and under the condition \eqref{a5} it holds 
 if $u_0\in H^{1,1/2}(\R)$ with $\|u_0\|_{H^1}<\eta=\eta(j;f_j)$ is sufficiently small. 
\end{rem} 
\begin{rem}
We observe that non-linearity $\,f(\cdot)\,$ in \eqref{a7} may involve $L^2$-sub-critical and $L^2$-super-critical parts. In this case, the properties \eqref{clas4} and \eqref{a9} can be guarantee by assuming that the initial small quantities $\,\epsilon_1=\|u_0\|_{L^2}$ and $\,\epsilon_2=\|\partial_xu_0\|_{L^2}$ satisfy appropriate relations, for example, $\epsilon_2>c\epsilon_1^3$. Under these assumption one also has that \eqref{a8} holds.

\end{rem}

\vskip.1in

\begin{rem} 
The proof of Theorems \ref{theo:1}-\ref{theo:4} are elementary and are based on the time evolution of the quantities
\begin{equation}
\label{key}
\Omega(t):=\int_{-\infty}^{\infty}\,x\,u(t,x) dx,\;\;\;\;\;\text{and}\;\;\;\;\;\;\Lambda(t):=\int_{-\infty}^{\infty}\,x\,u^2(t,x) dx,
\end{equation}
which are time periodic function if $\,u=u(t,x)$ is also time periodic.

\end{rem}

\medskip

 Our second goal is to establish precise decay in KdV-like equations ($k=2$ and $f_2\neq 0$). First, we review some results concerning decay of small solutions for gKdV equations. Kenig-Ponce-Vega \cite{KPV1} showed scattering for small data solutions of the $L^2$ critical gKdV equation ($p=5$ and $f_1=0$). Ponce and Vega \cite{PV} showed that for the case $f(s)=|s|^p$, $p>(9+\sqrt{73})/4\sim 4.39$, small data solutions in $L^{1}\cap H^{2}$ lead to decay, with rate $t^{-1/3}$ (i.e. linear rate of decay). Christ and Weinstein \cite{CW} improved  it to the case $f(s)= |s|^p$, $p>\frac14(23-\sqrt{57}) \sim 3.86$. Hayashi and Naumkin \cite{HN1,HN2} studied the case $p>3$, obtaining decay estimates and asymptotic profiles for small data in $\mathcal H^{1,1}$. Tao \cite{Tao} considered the scattering of data for the quartic KdV in the space $H^1 \cap \dot H^{-1/6}$ around the zero and the soliton solution. The finite energy condition was then removed by Koch and Marzuola in \cite{KM} by using $U-V$ spaces. C\^ote \cite{Cote} constructed solutions to the subcritical gKdV equations with a given asymptotical behavior, for $k=4$ and $k=5$.

\medskip

The case $k=3$ is critical in terms of scattering, and also critical with respect to the $L^1$ norm. Therefore, it is expected that small solutions do scatter following modified linear profiles. In \cite{GPR}, Germain, Pusateri and Rousset deal with the mKdV case ($k=3$) around the zero background and the soliton. By using Fourier techniques and estimates on space-time resonances, they were able to tackle down this ``critical'' case in terms of modified scattering. Using different techniques, related to Hardy type estimates, Kenig, Ponce and Vega \cite{KPV2} showed that solutions of gKdV decaying sufficiently fast must be identically zero. Finally, Isaza, Linares and Ponce \cite{ILP} showed bounds on the spatial decay estimates on \eqref{gKdV} with $k=2$ in sufficiently strong weighted Sobolev spaces. However, we recall that no scattering results as above mentioned seems to hold for the quadratic power ($k=2$), which can be considered as ``supercritical'' in terms of modified scattering, since the nonlinear term behaves (under the linear Airy decay dynamics $\sim t^{-1/3}$) as $\partial_x u u \sim \frac{1}{t^{2/3}} u$, and can be regarded as linear term with a potential which is clearly not integrable in time.
\medskip

Let $C>0$ be an arbitrary constant, and let $I(t)$ be the time-depending interval 
\be\label{Interval}
I(t):= \left( - \frac{C |t|^{1/2}}{\log |t|} , \frac{C |t|^{1/2}}{\log |t|} \right),  \quad |t|\geq 2.
\ee
Clearly $I(t)$ contains any compact interval for $t$ sufficiently large.
 
\begin{thm}\label{TH1}
Assume $k=2$ and $f_2(s)=0$ in \eqref{f(s)}. Suppose that $u=u(x,t)$ is a solution to \eqref{gKdV} such that
 \begin{equation}
 \label{m1}
 u\in C(\R:H^1(\R)) \cap L^\infty(\R:L^1(\R)),
\end{equation}
and there exists $\ve>0$ such that 
\be\label{Smallness0}
\sup_{t\in \R}\|u(t)\|_{H^1} <\ve.
\ee
Then, one has
\be\label{AS}
\lim_{t\to \pm\infty}\|u(t) \|_{ H^1(I(t))} =0.
\ee
Additionally, the following Kato-smoothing estimate holds: for any $c_0>0$,
\be\label{Integra50}
\int_2^\infty \!\! \int_{-\infty}^{\infty} e^{-c_0|x|} (u^2+ (\partial_x u)^2 + (\partial_x^2 u)^2)(t,x)dx dt <+\infty.
\ee
Moreover, no small soliton nor breather solution exists for KdV inside the region $I(t)$, for any time $t$ sufficiently large.
\end{thm}

\begin{rem}
Note that Theorem \ref{TH1} requires only small solutions in $H^1$, and $L^1$ control uniform-in-time. This requirement is nonempty: solutions satisfying this condition are e.g. small solitons and multi-solitons (any finite number), see \eqref{Soliton}. Note also that solitons move away (in finite time) from the region $I(t)$ considered in \eqref{AS}, and they do not decay. These assumptions should be compared with the ones required in Kenig, Ponce and Vega \cite{KPV2} to show nonexistence of solitary wave solutions.
\end{rem}

\begin{rem}
Theorem \ref{TH1} can be complement by the following standard fact: all small solutions uniformly bounded in time in $H^1$ satisfy the asymptotic regime \cite{MM2,AM1}
\[
\lim_{t\to +\infty} \|u(t)\|_{H^1(x>v t)} =0,
\]
where $v=v(\|u(t=0)\|_{H^1})$, and $v\to 0^+$ as $\|u(t=0)\|_{H^1}\to 0$. This last convergence result is usually referred as \emph{decay in the soliton region}. Under the regime proposed in Theorem \ref{TH1}, the remaining regions $\left( -\infty, -\frac{C |t|^{1/2}}{\log |t|} \right) \cup \left(  \frac{C |t|^{1/2}}{\log |t|} , v t \right)$ are probably strongly dominated by modified linear decay. Note in addition that linear KdV (i.e. Airy equation) possesses as dispersion relation $\omega(k)=-k^3$, and its group velocity is given by $\omega'(k)=-3k^2\leq 0$. Theorem \ref{TH1} formally shows that when the nonlinearity is turn on, the linear mode $k=0$ decays to zero.
\end{rem}

\begin{rem}
If the initial data $u_0$ satisfies $u_0\in  H^{1,1}(\R)$ and $u \in L^\infty(\R:H^{0,1}(\R))$,
then the hypotheses of Theorem \ref{TH1} are satisfied. Indeed, by conservation of mass and energy the corresponding solution $u(t)$ satisfies $\sup_{t\in \R}\|u(t)\|_{L^\infty} \lesssim 1$, and 
$$\|u(t)\|_{L^1} =\int_{|x|<R} |u(t)| + \int_{|x|\geq R} \frac{|x|}{|x|} |u(t)| \lesssim \|u(t)\|_{L^\infty} R + R^{-1/2} \|u(t)\|_{H^{0,1}} \lesssim  \|u(t)\|_{H^{0,1}}^{2/3}.$$ 

However,  the condition $u_0 \in H^{1,1}(\R)$ is not preserved by the flow solution since this would imply that  $u_0\in H^{2,1}(\R)$, see \cite{ILP}
Theorem 1.4. \end{rem}

%

\begin{rem}
The restriction in Theorem \ref{TH1} to the interval $|x| \sim |t|^{\frac12} \log^{-1}|t|$ is probably consequence of the initial assumptions made, and conclusions should be slightly different if the $L^1$ boundedness is not required. We conjecture that decay to zero in the energy space should hold in any compact set. See also the works \cite{AM3,MPP,KMPP,KM} for decay in extended regions of space which profit of internal directions for decay. 
\end{rem}

\begin{rem}[About the mKdV case]
Our result is not valid for the mKdV case $k=3$ because of the existence of standing breathers, namely \eqref{breather} with $\ga=0$, i.e. $\bt =\pm\sqrt{3} \al$. From \cite{AM}, it is well-known that small $H^1$ breathers are characterized by the constraints $\bt$ small (\lq\lq small mass'') and $\al^2 \bt$ also small (``small energy''). These two conditions are clearly not in contradiction with the additional assumption $\ga=0$. Therefore, even small mKdV $H^1$ breathers may have zero velocity and do not scatter to infinity. For further results on the mKdV equation under initial conditions in weighted spaces, see \cite{HN2,H,GPR}. See also \cite{AM,AM1,AM2} for more properties on mKdV breathers.
\end{rem}

\begin{rem}
Since KdV is completely integrable, Inverse Scattering Transform (IST) methods can be applied to obtain formal asymptotics of the solution, by assuming initial data smooth and rapidly decaying. See Ablowitz and Segur \cite[p. 80]{AS}, Deift, Venakides and Zhou. \cite{DVZ} and Eckhaus and Schuur \cite{ES} and references therein for more details on these methods. (See also \cite{KMM3} for a detailed description of results for the KdV problem.) In these references, the region $|x|\lesssim t^{1/3}$ and its complement play an important role in terms of describing different dynamics, but under suitable decay assumptions. In this work, the $L^1$ boundedness and small $H^1$ conditions lead to a rigorous proof of decay in the larger region $|x|\lesssim t^{1/2}$, working even for integrable and non-integrable modifications of KdV. We recall that Theorem \ref{TH1} does not require the ``no soliton'' hypothesis, and it is proven with data only in a subset of the energy space. 
\end{rem}

\begin{rem}
Theorem \ref{TH1} is false if $u$ is assumed complex-valued. Indeed, $u(t,x):= \frac{-6}{(x+ i\ve)^2}$, for any real-valued $\ve\neq 0$, is a solution to KdV and counterexample to Theorem \ref{TH1}.
\end{rem}

An important advantage of the methods of proof is that we can extend Theorem \ref{TH1} to the case of KdV perturbations, including the integrable Gardner equation, provided the initial data is small in the $H^1$ norm.

\begin{thm}\label{TH2}
Consider $k=2$ and $f_2(s)\neq 0$ in \eqref{f(s)}. Assume that $u=u(t,x)$ is a solution to \eqref{gKdV} such that
\begin{equation}
\label{m2}
u\in C(\R:H^1(\R)) \cap L^\infty(\R:L^1(\R))
\end{equation}
  and there exists $\ve>0$ such that 
\be\label{Smallness}
\sup_{t\in \R}\|u(t)\|_{H^1} <\ve.
\ee
Then, given the same time-depending interval $I(t)$ as in \eqref{Interval}, one has \eqref{AS} and \eqref{Integra50}. In particular, all Gardner-like equations in \eqref{gKdV} cannot possess small, standing breather-like solutions.
\end{thm}

\begin{rem}

Theorem \ref{TH2} gives another proof of the result in Theorem \ref{theo:2} for the case $k=2$ with a more precise information on the asymptotic behavior of the global small solutions. 

\end{rem}

\begin{rem}
Note that condition \eqref{Smallness} is in certain sense necessary, because large solutions $u(t)$ formally follow a dynamics in which $f_2(u(t))$ is larger than $u^2(t)$, and a completely different dynamics could appear from these assumptions; in particular, the main term could be of cubic order, just like mKdV, which has breather solutions. 
\end{rem}

\begin{rem}\label{Gardner_remark}
Theorem \ref{TH2} may be in contradiction with the existence of \emph{Gardner breathers}. However, under the hypothesis \eqref{Smallness} of Theorem \ref{TH2}, this is not the case.
Being in \eqref{BGe}-\eqref{BGe2} $\mu>0$ a fixed parameter, small energy breathers must have $\bt$ small (cf. \cite[Lemmas 2.2 and 2.6]{Alejo}), and standing breathers must obey the zero speed condition $\ga=0$, which implies $\al$ also small. Under these two conditions, $\Delta=\al^2+\bt^2-\frac2{9\mu}$ cannot be positive and breathers are not well-defined. Note that $\ga\neq 0$ implies that breathers leave the region $|x| \sim t^{1/2}$ very fast and therefore Theorem \ref{TH2} can be understood as a sort of ``nonlinear scattering'' in Theorem \ref{TH2}. See e.g. \cite{AMV,Alejo} for the role of the Gardner equation in deciding several long-time properties and behavior of the original KdV equation.
\end{rem}

 The proof of Theorems \ref{TH1} and \ref{TH2} are by itself elementary, and no extensive nor deep Fourier analysis are required. We prove this result by following very recent developments concerning the decay of solutions in  $1+1$ dimensional scalar field models. Kowalczyk, Martel and the author of this note showed in \cite{KMM,KMM1} that well chosen Virial functionals can describe in great generality the decay mechanism for models where standard scattering is not available (i.e. there is modified scattering), either because the dimension is too small, or the nonlinearity is long range. Moreover, this decay mechanism also describes \lq\lq nonlinear scattering'', in the sense that solutions like \eqref{Soliton} are also discarded by Theorem \ref{TH1} inside $I(t)$, $t\to+\infty$. We will also use as advantage the \emph{subcritical} character of the $L^1$ norm for KdV and its perturbations. This give us control on the $L^2$ norm of the solution before arriving to the standard virial identity in the energy space, which is very hard to control by itself without a previous control on the local $L^2$ norm.

\medskip

  Previous Virial-type decay estimates 
were obtained by Martel-Merle and Merle-Rapha\"el \cite{MM,MR} in the case of generalized KdV and nonlinear Schr\"odinger equations. Moreover, 
the results proved in this note and in \cite{KMM,KMM1,MM,MR} apply to equations which have long range nonlinearities, as well as very low or null 
decay rates. See also \cite{MPP,KMPP} for other applications of this technique to the case of Boussinesq equations, and \cite{KM1} for an application to the case of the BBM equation. 
\medskip

The rest of this paper is organized as follows: The proof of Theorems \ref{theo:1}-\ref{theo:2} will be given in section 2. Theorems \ref{theo:3}-\ref{theo:4} will be proven in section 3.
Section 4 contains the proofs of Theorems \ref{TH1} and \ref{TH2} which is divided in three  steps.
\medskip

\subsection*{Acknowledgments} We are indebted to M.A. Alejo for several interesting comments and remarks about a first version of this work.

\bigskip

\section{Proof of Theorems \ref{theo:1} and  \ref{theo:2}. }

The proof of Theorems \ref{theo:1} and \ref{theo:2} are direct consequences of the following result:
\begin{prop} \label{prop:1} 
If
\begin{equation}
\label{class1}
u\in C(\R:H^{4,2}(\R)) 
\end{equation}
is a strong solution of the IVP \eqref{gKdV}, then
\begin{equation}
\label{con1}
\frac{d}{dt}\int_{-\infty}^{\infty}xu(t,x)dx=\int_{-\infty}^{\infty}f(u(t,x))dx.
\end{equation}
\end{prop}

\begin{proof} 
Multiplying the equation in \eqref{gKdV} by $x$ and integrating the result in $\R$ one gets \eqref{con1}.
\end{proof}

\vskip.1in

\begin{proof}[Proof of Theorem \ref{theo:1}]

If $u=u(x,t)$ is time periodic with period $\omega$ integrating \eqref{con1} one obtains \eqref{A1}.

Inserting the assumption \eqref{a1} into \eqref{con1} one gets that
\begin{equation}
\label{d1}
\Omega(t)=\int_{-\infty}^{\infty}xu(t,x)dx
\end{equation}
is strictly monotonic so it cannot be periodic. This yields the desired result.
\end{proof} 

\vskip.1in

\begin{proof}[Proof of Theorem \ref{theo:2}]

The hypotheses \eqref{a2} and \eqref{a3} guarantee that \
\[
\,f(u(t,x))\geq 0,\;\;\;\;\;(t,x)\in\R^2. 
\]
Hence, the argument in the proof of Theorem \ref{theo:1} completes the proof.
\end{proof} 

\medskip

\section{Proof of Theorems \ref{theo:3} and \ref{theo:4}. }

The proof of Theorems \ref{theo:3} and \ref{theo:4} are direct consequences of the following result:
\begin{prop} \label{prop:2} 
If
\begin{equation}
\label{classs1}
u\in C(\R:H^{1,1/2}(\R)) 
\end{equation}
is a strong solution of the IVP \eqref{gKdV}, then
\begin{equation}
\label{con2}
\frac{d}{dt}\int_{-\infty}^{\infty}xu^2(t,x)dx=-6\left(I_3(u_0)+\int_{-\infty}^{\infty}\Big(G(u)-\frac{F(u)}{3}\,\Big)(t,x)dx\right),
\end{equation}
where $G$ is as in \eqref{cl} and 
\begin{equation}
\label{a4}
\partial_xF(u):=u\partial_xf(u),\;\;\;\;\;\;F(0)=0.
\end{equation}
\end{prop}

\begin{proof} 
Multiplying the equation in \eqref{gKdV} by $xu(t,x) $, integrating the result in $\R$ and using the conservation law  $I_3(u)$ in \eqref{CL} one obtains \eqref{con2}.
\end{proof}

\begin{rem}
We observe that in the case of the mKdV, i.e. $f(u)=u^{3}$ in \eqref{gKdV} one has from \eqref{con2}
\[
\frac{d}{dt}\int_{-\infty}^{\infty}xu^2(t,x)dx=-6\,I_3(u_0).
\]

Therefore, for the periodic breather solution $B=B(t,x)$ of the mKdV, i.e. $\gamma=0$ in  \eqref{breather}, one has  that
\[
I_3(B)=0.
\]
\end{rem}

\begin{proof}[Proof of Theorem \ref{theo:3}]

First, we recall  a particular case of the Gagliardo-Nirenberg inequality : for any $\,p\in[2,\infty]$
\begin{equation}
\label{GN}
\|u(t) \|_{p}\leq c_p\,\|\partial_xu(t)\|_2^{1/2-1/p}\,\|u(t)\|_2^{1/2+1/p}.
\end{equation}

\vskip.05in

Combining the assumptions \eqref{a5}-\eqref{a6}  and \eqref{GN} one gets that
\begin{equation}
\label{d4}
\begin{aligned}
&I_3(u_0)+\int_{-\infty}^{\infty}\Big(G(u)-\frac{F(u)}{3}\,\Big)(t,x)dx\\
& \geq \frac12 \|\partial_xu(t)\|_2^2 -c\|u_0\|_2^4\,\|\partial_xu(t)\|_2^2\geq  \frac14\|\partial_xu(t)\|_2^2.
\end{aligned}
\end{equation}

Therefore, inserting \eqref{d4} in \eqref{con2} it follows that
\begin{equation}
\label{d3}
\Lambda(t)=\int_{-\infty}^{\infty}xu^2(t,x)dx
\end{equation}
is strictly decreasing so it cannot be periodic.
\end{proof} 

\vskip.1in

\begin{proof}[Proof of Theorem \ref{theo:4}]

Since 
\begin{equation}
\label{a7b}
f(u)=u^{3}+\beta u^5+f_{5}(u),  \;\;\;\;\beta<0,
\end{equation}
we have that
\begin{equation}
\label{d5}
\begin{aligned}
\begin{cases}
&F(u) =\frac{3}{4}\,u^4+\frac{5\beta}{6}\,u^6+F_{f_5},\\
\\
&G(u)=u^4+\frac{\beta}{6}\,u+G_{f_5},
\end{cases}
\end{aligned}
\end{equation}
with

\begin{equation}
\label{d6}
\begin{aligned}
\begin{cases}
&\partial_xF_{f_5}(u) =u\,\partial_xf_5(u),\;\;\;\;\;F_{f_5}(0)=0,\\
\\
&G_{f_5}(u)=\int_0^uf_5(s)ds.
\end{cases}
\end{aligned}
\end{equation}

Therefore,
\begin{equation}
\label{d7}
G(u)-\frac{1}{3}F(u)=\frac{\beta}{6} u^6\Big(1-\frac{5}{3}\Big)+G_{f_5}(u)-\frac{1}{3}F_{f_5}(u),
\end{equation}
with
\begin{equation}
\label{d8}
|G_{f_5}(s)|+|F_{f_5}(s)|=o(|s|^6)\;\;\;\;\;\;\text{as}\;\;\;\;\;|s|\to 0.
\end{equation}

Combining \eqref{d7}, \eqref{d8} and the hypotheses \eqref{a7}-\eqref{a8} from \eqref{con2}
one conclude that
\begin{equation}
\label{d9}
\frac{d}{dt}\int_{-\infty}^{\infty}xu^2(t,x)dx \lneq 0,
\end{equation}
which yields the desired result.

\end{proof}

\medskip

\section{Proof of Theorems \ref{TH1} and \ref{TH2}}

\noindent
\underline{ Step 1:} Integrability of the local $L^2$ norm.
\medskip

\noindent
With no loss of generality, we assume now that $t\geq 2$ and $C=1$ in \eqref{Interval}. Let
\be\label{la}
\la(t) := \frac{t^{1/2}}{\log t}.
\ee
Note that
\[
\la'(t) = \frac1{2t^{1/2} \log t} -\frac1{t^{1/2} \log^2 t} = \frac1{t^{1/2} \log t} \left(\frac12 - \frac1{\log t}\right),
\]
and
\be\label{Computations}
\frac{\la'(t)}{\la(t)} = \frac{1}{t}\left( \frac12 - \frac1{\log t} \right), \quad \la'^2(t) =  \frac{1}{t \log^2 t}\left(\frac12 - \frac 1{\log t} \right)^2.
\ee
Let $\psi$ be a smooth bounded weight. Let us consider the functional for $u$ solving \eqref{gKdV}
\be\label{III}
\mathcal I(t):=\int_{-\infty}^{\infty} \psi \Big( \frac{x}{\la(t)} \Big)  u(t,x)dx.
\ee
Notice that from the hypothesis of Theorems \ref{TH1}-\ref{TH2}, one has  $\sup_{t\in \R}\mathcal I(t) <+\infty$. We claim the following result, whose proof is direct. 

\begin{lem}\label{dtI0}
We have
\be\label{dtI}
\begin{aligned}
\frac d{dt}\mathcal I(t) = &~ {} - \frac{\la'(t)}{\la(t)} \int \frac{x}{\la(t)}  \psi'\Big( \frac{x}{\la(t)} \Big) u(t,x)dx 
+ \frac1{\la^3(t)}\int \psi^{(3)}\Big( \frac{x}{\la(t)} \Big) u(t,x)dx \\
& ~ {} + \frac{1}{\la(t)} \int \psi ' \Big( \frac{x}{\la(t)} \Big) (u^2 + f_2(u))(t,x)dx.
\end{aligned}
\ee
\end{lem}

The main result of this section is an averaged control of the local $L^2$ norm of the solution at infinity in time.

\begin{cor}\label{Integra2}
Consider $\psi(x) := \tanh x$, such that $\psi'(x) = \sech^2 x$. Assume that $k=2$ and either $f_2(s) =0$, or $f_2(s)\neq 0$ and condition \eqref{Smallness} holds. Then we have
\be\label{Integra0}
\int_2^\infty\frac1{\la(t)} \int_{-\infty}^{\infty} \sech^2 \Big( \frac{x}{\la(t)} \Big)  u^2(t,x)dx dt < +\infty,
\ee
and for some increasing  sequence of time $t_n\to +\infty$,
\be\label{Integra1}
\lim_{n\to \infty} \int_{-\infty}^{\infty} \sech^2 \Big( \frac{x}{\la(t_n)} \Big) u^2(t_n,x)dx =0.
\ee
\end{cor}

\begin{proof}
Let us estimate the terms in \eqref{dtI}. With the choice of $\la(t)$ given in \eqref{la}, we have
\[
\begin{aligned}
\abs{ \frac{\la'(t)}{\la(t)} \int \frac{x}{\la(t)} \psi'\Big( \frac{x}{\la(t)} \Big) u(t,x)dx} \le &~  \frac{(\la'(t))^2}{\la(t)} \int \Big(\frac{x}{\la(t)}\Big)^2 \psi'\Big( \frac{x}{\la(t)} \Big) dx \\
& ~ +\frac{\la'(t)}{4\la(t)} \frac1{\la'(t)} \int  \psi'\Big( \frac{x}{\la(t)} \Big) u^2(t,x)dx.
\end{aligned}
\]
Therefore
\be\label{Est1}
\abs{ \frac{\la'(t)}{\la(t)} \int \frac{x}{\la(t)} \psi'\Big( \frac{x}{\la(t)} \Big) u(t,x)dx} \le  C(\la'(t))^2 +\frac{1}{4\la(t)}\int  \psi'\Big( \frac{x}{\la(t)} \Big) u^2(t,x)dx.
\ee
On the other hand, using \eqref{la}
\be\label{Est2}
\begin{aligned}
 \frac1{\la^3(t)}\int \psi^{(3)}\Big( \frac{x}{\la(t)} \Big) u(t,x)dx  \lesssim  \frac1{\la^3(t)}  \la^{1/2}(t) \|u(t)\|_{L^2} 
 \lesssim     \frac{\log^{5/2} t}{t^{5/4}} .
\end{aligned} 
\ee
Hence, in the case where $f_2(s) =0$, we conclude from \eqref{dtI}, \eqref{Est1} and \eqref{Est2} that 
\[
\begin{aligned}
 \frac{1}{\la(t)} \int \psi ' \Big( \frac{x}{\la(t)} \Big) u^2(t,x)dx  \leq &~ \frac{C\log^{5/2} t}{t^{5/4}}+ C(\la'(t))^2 \\
 & \quad ~ {} +\frac{1}{4\la(t)}\int  \psi'\Big( \frac{x}{\la(t)} \Big) u^2(t,x)dx.
\end{aligned}
\]
Using that $\psi '=\sech^2$ and \eqref{Computations}, we get
\[
 \frac{1}{\la(t)} \int\sech^2 \Big( \frac{x}{\la(t)} \Big) u^2(t,x)dx  \lesssim  \frac{\log^{5/2} t}{t^{5/4}} +  \frac{1}{t \log^2 t},
\]
an estimate that shows \eqref{Integra0} for this case. For the case $f_2(s) \neq 0$, we use \eqref{f(s)}  to bound 
\[
s^2 + f_2(s) \geq \frac12 s^2,
\]
valid for all $|s|$ small. Consequently, from \eqref{Smallness},
\[
 \frac{1}{\la(t)} \int \psi ' \Big( \frac{x}{\la(t)} \Big) (u^2 + f_2(u))(t,x)dx \geq  \frac{1}{2\la(t)} \int \psi ' \Big( \frac{x}{\la(t)} \Big) u^2 (t,x)dx,
\]
and the rest of the proof is the same as in the $f_1(s)=0$ case. Finally, \eqref{Integra1} is just consequence of \eqref{Integra0} and the fact that $\frac1{\la(t)}$ is not integrable in $[2,\infty)$.
\end{proof}

\noindent
\underline{ Step 2:} Averaged local $L^2$ decay of the derivatives.

\medskip
\noindent
Let $\phi$ be a smooth, bounded function, to be defined later. Let us define the functional
\be\label{J}
\mathcal J(t):= \frac12\int_{-\infty}^{\infty} \phi\Big( \frac{x}{\la(t)} \Big) u^2(t,x)dx.
\ee
Notice that from the hypothesis of Theorems \ref{TH1} and \ref{TH2}, one has  $\sup_{t\in \R}\mathcal J(t) <+\infty$. We claim the following result.

\begin{lem}\label{dtJ0}
Let $F_2:\R\to \R$ be such that $F_2(0)=0$ and $F_2'(s) =f_2(s)$ (see \eqref{f(s)}). Then we have
\be\label{dtJ}
\begin{aligned}
\frac d{dt}\mathcal J(t) = &~ {} - \frac{\la'(t)}{2\la(t)} \int \frac{x}{\la(t)} \phi'\Big( \frac{x}{\la(t)} \Big) u^2(t,x)dx -\frac3{2\la(t)} \int \phi'\Big( \frac{x}{\la(t)} \Big) (\partial_x u)^2(t,x)dx \\
& ~ {}+ \frac1{2\la^3(t)}\int \phi^{(3)}\Big( \frac{x}{\la(t)} \Big) u^2(t,x)dx \\
& ~ {}+ \frac{1}{\la(t)} \int \phi ' \Big( \frac{x}{\la(t)} \Big) \Big( \frac23u^{3} + uf_2(u) - F_2(u)\Big)(t,x)dx.
\end{aligned}
\ee
\end{lem}

\begin{proof}
The proof of this result is standard, see e.g. \cite{MM2} for details.
\end{proof}

Collecting the results obtained in Corollary \ref{Integra2} and Lemma \ref{dtJ0}, we conclude that the local $L^2$ norm of the derivate of $u$ is integrable in time. 
\begin{cor}
Consider the weight function $\phi(x) := \tanh (2x)$, such that $\phi'(x) = 2\sech^2 (2x)$.  Then we have from \eqref{dtJ}
\be\label{Integra4}
\int_2^\infty\frac1{\la(t)} \int_{-\infty}^{\infty} \sech^4 \Big( \frac{x}{\la(t)} \Big) (\partial_x u)^2(t,x)dx dt <+\infty.
\ee
In particular, there exists an increasing sequence of times $s_n\to +\infty$ such that 
\be\label{AS1}
\lim_{n\to +\infty}\int_{-\infty}^{\infty} \sech^4 \Big( \frac{x}{\la(s_n)} \Big) (\partial_x u)^2(s_n,x)dx =0.
\ee
\end{cor}

\begin{proof}
Similar to the proof of Corollary \ref{Integra2}. See also \cite[Lemma 2.1]{MPP} for a similar proof. The only difficult term is the nonlinear, which is bounded as follows:
\[
\abs{ \frac{1}{\la(t)} \int \phi ' \Big( \frac{x}{\la(t)} \Big) \Big( \frac23u^{3} + uf_2(u) - F_2(u)\Big)(t,x)dx} \lesssim \frac{\ve}{\la(t)} \int \phi ' \Big( \frac{x}{\la(t)} \Big) u^2.
\]
Using Corollary \ref{Integra2}, this quantity integrates in time.
\end{proof}

Now we prove the $L^2$ decay in \eqref{AS}. Recall the interval $I(t)$ defined in \eqref{Interval}.
\begin{cor}
Assume the hypothesis of Theorem \ref{TH1}, or Theorem \ref{TH2}. Then  
\be\label{AS0}
\lim_{t\to \pm\infty}\|u(t) \|_{L^2(I(t))} =0.
\ee
\end{cor}

\begin{proof}
It is enough to prove the result in the case $t\geq 2$. Consider now the weight function $\phi(x) := \sech^6 x$  in \eqref{J}. Then, \eqref{dtJ} leads to the estimate (see \cite{MPP} for similar results)
\be\label{dtK_estimate}
\abs{\frac d{dt}\mathcal J(t)} \lesssim \frac1{\la(t)}\int \sech^4\Big( \frac{x}{\la(t)} \Big) (u^2 +(\partial_x u)^2) (t,x)dx.
\ee
This estimate, \eqref{Integra0} and \eqref{Integra4} imply that for $t<t_n$,
\[
\abs{\mathcal J(t) -\mathcal J(t_n) } =  o(1)\;\;\;\;\;\;\text{as}\;\;\;\;\;t\uparrow \,\infty,
\]
and using \eqref{Integra1} we conclude.
\end{proof}

\noindent
Consider again $\phi$ smooth and bounded. Let us define now the functional
\be\label{K}
\mathcal K(t):= \frac12\int_{-\infty}^{\infty} \phi\Big( \frac{x}{\la(t)} \Big) \left((\partial_xu)^2 -\frac23u^3 -2F_2(u)\right)(t,x)dx.
\ee
Notice that from the hypothesis of Theorems \ref{TH1} and \ref{TH2}, one has  $\sup_{t\in \R}\mathcal{K}(t) <+\infty$. We claim the following result.

\begin{lem}
We have
\be\label{dtK}
\begin{aligned}
\frac d{dt}\mathcal{K}(t)= &~ {} - \frac{\la'(t)}{2\la(t)} \int \frac{x}{\la(t)} \phi' \Big( \frac{x}{\la(t)} \Big) \left((\partial_xu)^2 -\frac23u^3 -2F_2(u)\right) \\
& ~ {} -\frac3{2\la(t)}  \int  \phi'\Big( \frac{x}{\la(t)} \Big) (\partial_x^2 u)^2  +  \frac1{2\la^3(t)} \int   \phi'''\Big( \frac{x}{\la(t)} \Big)(\partial_x u)^2  \\
& ~ {}  +  \frac1{3\la^3(t)} \int   \phi'''\Big( \frac{x}{\la(t)} \Big) u^3   + \frac1{\la^3(t)} \int   \phi'''\Big( \frac{x}{\la(t)} \Big)  F_2(u)  \\
& ~ {} + \frac2{\la(t)} \int \phi'\Big( \frac{x}{\la(t)} \Big) \partial_x( u^2 + f_2(u))   \partial_x u \\
& ~ {} + \frac2{\la^2(t)} \int \phi''\Big( \frac{x}{\la(t)} \Big) ( u^2 + f_2(u))   \partial_x u \\
& ~ {} - \frac1{2\la(t)} \int \phi'\Big( \frac{x}{\la(t)} \Big) ( u^2 + f_2(u))^2 .
\end{aligned}
\ee
\end{lem}

\begin{proof}
We compute: first,
\[
\begin{aligned}
\frac d{dt}\mathcal{K}(t)= &~ {} - \frac{\la'(t)}{2\la(t)} \int \frac{x}{\la(t)} \phi' \Big( \frac{x}{\la(t)} \Big) \left((\partial_xu)^2 -\frac23u^3 -2F_2(u)\right) \\
& ~ {} +  \int  \phi\Big( \frac{x}{\la(t)} \Big) (\partial_xu  \partial_{tx}u - u^2 \partial_tu -f_2(u)\partial_tu).
\end{aligned}
\]
Integrating by parts, 
\[
\begin{aligned}
\frac d{dt}\mathcal{K}(t)= &~ {} - \frac{\la'(t)}{2\la(t)} \int \frac{x}{\la(t)} \phi' \Big( \frac{x}{\la(t)} \Big) \left((\partial_xu)^2 -\frac23u^3 -2F_2(u)\right) \\
& ~ {} -  \int  \phi\Big( \frac{x}{\la(t)} \Big) (\partial_x^2 u  + u^2  + f_2(u))\partial_tu  - \frac1{\la(t)} \int  \phi'\Big( \frac{x}{\la(t)} \Big) \partial_xu  \partial_tu ,
\end{aligned}
\]
and using the equation,
\[
\begin{aligned}
\frac d{dt}\mathcal{K}(t)= &~ {} - \frac{\la'(t)}{2\la(t)} \int \frac{x}{\la(t)} \phi' \Big( \frac{x}{\la(t)} \Big) \left((\partial_xu)^2 -\frac23u^3 -2F_2(u)\right) \\
& ~ {} -\frac1{2\la(t)}  \int  \phi'\Big( \frac{x}{\la(t)} \Big) (\partial_x^2 u  + u^2  + f_2(u))^2  \\
& ~{} - \frac1{\la(t)} \int  \partial_x\left( \phi'\Big( \frac{x}{\la(t)} \Big) \partial_xu \right) (\partial_x^2 u  + u^2  + f_2(u)) .
\end{aligned}
\]
Simplifying,
\[
\begin{aligned}
\frac d{dt}\mathcal{K}(t)= &~ {} - \frac{\la'(t)}{2\la(t)} \int \frac{x}{\la(t)} \phi' \Big( \frac{x}{\la(t)} \Big) \left((\partial_xu)^2 -\frac23u^3 -2F_2(u)\right) \\
& ~ {} -\frac1{2\la(t)}  \int  \phi'\Big( \frac{x}{\la(t)} \Big) (\partial_x^2 u  + u^2  + f_2(u))^2  \\
& ~ {} - \frac1{\la^2(t)} \int   \phi''\Big( \frac{x}{\la(t)} \Big) \partial_xu (\partial_x^2 u  + u^2  + f_2(u)) \\
& ~ {} - \frac1{\la(t)} \int \phi'\Big( \frac{x}{\la(t)} \Big) (\partial_x^2 u  + u^2  + f_2(u) -u^2 -f_2(u))   (\partial_x^2 u  + u^2  + f_2(u)) ,
\end{aligned}
\]
and
\[
\begin{aligned}
\frac d{dt}\mathcal{K}(t)= &~ {} - \frac{\la'(t)}{2\la(t)} \int \frac{x}{\la(t)} \phi' \Big( \frac{x}{\la(t)} \Big) \left((\partial_xu)^2 -\frac23u^3 -2F_2(u)\right) \\
& ~ {} -\frac3{2\la(t)}  \int  \phi'\Big( \frac{x}{\la(t)} \Big) ((\partial_x^2 u)^2  + (u^2  + f_2(u))^2 + 2\partial_x^2 u  ( u^2  + f_2(u) ))  \\
& ~ {}   + \frac1{2\la^3(t)} \int   \phi'''\Big( \frac{x}{\la(t)} \Big)(\partial_x u)^2  +  \frac1{3\la^3(t)} \int   \phi'''\Big( \frac{x}{\la(t)} \Big) u^3  \\
& ~ {} + \frac1{\la^3(t)} \int   \phi'''\Big( \frac{x}{\la(t)} \Big)  F_1(u)  \\
& ~ {} + \frac1{\la(t)} \int \phi'\Big( \frac{x}{\la(t)} \Big) ( u^2 + f_2(u))   \partial_x^2 u + \frac1{\la(t)} \int \phi'\Big( \frac{x}{\la(t)} \Big) ( u^2 + f_2(u))^2 ,
\end{aligned}
\]
which gives \eqref{dtK} after some integration by parts.
\end{proof}

Note that all the quantities in the RHS of \eqref{dtK} are known to be integrable in time using previous results, except the term with two derivatives $-\frac3{2\la(t)}  \int  \phi'\Big( \frac{x}{\la(t)} \Big) (\partial_x^2 u)^2$. An important corollary of this previous result is the following Kato-type smoothing estimate: 
\begin{cor}
Consider the weight function $\phi(x) := \tanh (3x)$, so that $\phi'(x) = 3\sech^2 (3x)$. Then we have 
\be\label{Integra5}
\int_2^\infty\frac1{\la(t)} \int_{-\infty}^{\infty} \sech^6 \Big( \frac{x}{\la(t)} \Big) (\partial_x^2 u)^2(t,x)dx dt <+\infty.
\ee
\end{cor}
This last estimate is proved using \eqref{Integra0} and \eqref{Integra4}. By taking $\la(t)=c_0$, and using \eqref{Integra0} and \eqref{Integra4}, we also have \eqref{Integra50}.

\bigskip

\noindent
 \underline{ Step 3 :} End of proof of Theorems \ref{TH1} and \ref{TH2}.

\medskip

Consider the weight $\phi(x):= \sech^{8}(x)$ in \eqref{K}. From \eqref{dtK} we have
\[
\abs{\frac d{dt}\mathcal K(t)} \lesssim \frac1{\la(t)}\int \sech^6\Big( \frac{x}{\la(t)} \Big) (u^2 +(\partial_x u)^2 + (\partial_x^2 u)^2) (t,x)dx.
\]
This estimate, \eqref{Integra0}, \eqref{AS1} and \eqref{Integra5} imply that for $t<s_n$,
\[
\abs{\mathcal K(t) -\mathcal K(s_n) } = o(1)\;\;\;\;\;\;\text{as}\;\;\;\;\;t\uparrow \,\infty,
\]
and using \eqref{AS0}, \eqref{AS1} and \eqref{Smallness0} or  \eqref{Smallness} we conclude.

\end{document}